\documentclass{amsart}
\usepackage{amsmath}
\usepackage{mathrsfs}
\usepackage{amssymb}
\usepackage{mathtools}
\usepackage{amscd}
\usepackage[all]{xy}
\usepackage{color}
\usepackage{fancyhdr}

\theoremstyle{plain}
\newtheorem{theorem}{Theorem}[section]
\newtheorem{proposition}[theorem]{Proposition}
\newtheorem{lemma}[theorem]{Lemma}
\newtheorem{corollary}[theorem]{Corollary}

\theoremstyle{definition}
\newtheorem*{definition}{Definition}

\theoremstyle{remark}
\newtheorem{remark}[theorem]{Remark}
\newtheorem{example}[theorem]{Example}

\def\Ext{{\rm Ext}}

\def\ker{{\rm ker}}
\def\im{{\rm im}}
\def\cok{{\rm coker}}

\def\wmax{w\mbox{-}{\rm Max}}
\def\Hom{{\rm Hom}}

\def\GV{{\rm GV}}

\def\:{\mathop{:}\limits}
\def\tor{{\rm tor_{\rm GV}}}

\begin{document}
\title[A new version of a theorem of Kaplansky]
{A new version of a theorem of Kaplansky}

\author [F. Wang] {Fanggui Wang}
\address{(Fanggui Wang) School of Mathematical Sciences, Sichuan Normal University, Chengdu, Sichuan 610066, China}
\email{wangfg2004@163.com}

\author [L. Qiao] {Lei Qiao$^\sharp$}
\address{(Lei Qiao) School of Mathematical Sciences, Sichuan Normal University, Chengdu, Sichuan 610066, China}
\email{lqiao@sicnu.edu.cn}


\thanks{$2010$ Mathematics Subject Classification: 13C13; 13D07; 13D30}

\thanks{$\sharp$ Address correspondence to Dr. Lei Qiao, School of Mathematical Sciences, Sichuan Normal University,
Chengdu, Sichuan 610066, China; E-mail: lqiao@sicnu.edu.cn}

\date{\today}

\begin{abstract} A well-known theorem of Kaplansky states that any
projective module is a direct sum of countably generated modules. In
this paper, we prove the $w$-version of this theorem, where $w$ is a
hereditary torsion theory for modules over a commutative ring.
\end{abstract}

\maketitle

\section{Introduction}

A well-known theorem of Kaplansky states that any projective module
is a direct sum of countably generated modules (see \cite{K58}).
This is equivalent to saying that every projective module can be
filtered by countably generated and projective modules. In
\cite{ST09}, \v{S}\v{t}ov\'{\i}\v{c}ek and Trlifaj applied Hill's
method \cite{H81} to extend Kaplansky's theorem on projective
modules to the setting of cotorsion pairs. Later, Enochs et al.
\cite{EEI14} also got the analogous version of Kaplansky's Theorem
for cotorsion pairs for a more general setting on concrete
Grothendieck categories. Moreover, several versions of Kaplansky's
theorem have been discussed in the literature. For example, a
categorical version of Kaplansky's theorem on projective modules is
proved in \cite[Lemma 3.8]{O78} by Osofsky. Also, in \cite{EAO13},
Estrada et al. prove a version of Kaplansky's Theorem for
quasi-coherent sheaves, by using Drinfeld's notion of almost
projective module and the Hill Lemma.

The purpose of this article is to present a $w$-version of
Kaplansky's theorem on projective modules, where $w$ is a hereditary
torsion theory for modules over a commutative ring. Next, we shall
review some terminology related to the hereditary torsion theory
$w$, see \cite{WK16} for details. Throughout, $R$ denotes a
commutative ring with an identity element and all modules are
unitary.

Recall from \cite{YWZC11} that an ideal $J$ of $R$ is called a
\textit{Glaz-Vasconcelos ideal} (a \textit{$\GV$-ideal} for short)
if $J$ is finitely generated and the natural homomorphism
$$\varphi:R\rightarrow J^*:=\Hom_R(J,R)$$ is an isomorphism. Notice that
the set $\GV(R)$ of GV-ideals of $R$ is a multiplicative system of
ideals of $R$. Let $M$ be an $R$-module. Define $$\tor(M):=\{x\in
M~|~\mbox{$Jx=0$ for some $J\in\GV(R)$}\}.$$ Thus $\tor(M)$ is a
submodule of $M$. Now $M$ is said to be \textit{$\GV$-torsion}
(resp., \textit{$\GV$-torsionfree}) if $\tor(M)=M$ (resp.,
$\tor(M)=0$). A GV-torsionfree module $M$ is called a
\textit{$w$-module} if $\Ext^1_R(R/J,M)=0$ for all $J\in\GV(R)$.
Then projective modules and reflexive modules are both $w$-modules.
In \cite[Theorem 6.7.24]{WK16}, it is shown that all flat modules
are $w$-modules. Also it is known that a GV-torsionfree $R$-module
$M$ is a $w$-module if and only if $\Ext^1_R(N,M)=0$ for every
GV-torsion $R$-module $N$ (see \cite[Theorem 6.2.7]{WK16}). For any
GV-torsionfree module $M$,
$$M_w:=\{x\in E(M)~|~\mbox{$Jx\subseteq M$ for some $J\in\GV(R)$}\}$$
is a $w$-submodule of $E(M)$ containing $M$ and is called the
$w$-\textit{envelope} of $M$, where $E(M)$ denotes the injective
envelope of $M$. It is clear that a GV-torsionfree module $M$ is a
$w$-module if and only if $M_w=M$.

It is worthwhile to point out that from a torsion-theoretic point of
view, the notion of $w$-modules coincides with that of $\tor$-closed
(i.e., $\tor$-torsionfree and $\tor$-injective) modules, where
$\tor$ is the torsion theory whose torsion modules are the
GV-torsion modules and torsionfree modules are the GV-torsionfree
modules, that is, the pair of classes of $R$-modules
$$\tor=\mbox{(\{GV-torsion modules\},\{GV-torsionfree modules\})}$$
is a hereditary torsion theory on the category of $R$-modules. In
the integral domain case, $w$-modules were called {\it
semi-divisorial modules} in \cite{GV77} and (in the ideal case)
$F_\infty$-{\it ideals} in \cite{HH80}, which have been proved to be
useful in the study of multiplicative ideal theory and module
theory.

In \cite{WK15}, the first named author and Kim generalized
projective modules to the hereditary torsion theory $\tor$ setting and
introduced the notion of $w$-projective modules.
Recall that an $R$-module $M$ is said to be \textit{$w$-projective}
if $\Ext^1_R(L(M),N)$ is GV-torsion for any torsionfree
$w$-module $N$, where $L(M)=\left(M/\tor(M)\right)_w$. It is clear
that both $\GV$-torsion modules and projective modules are
$w$-projective. Actually, the notion of $w$-projective modules
appeared first in \cite{W97} when $R$ is an integral domain. Thus,
it is natural to ask if Kaplansky's theorem on projective modules
has a $w$-module theoretic analogue.

To give a $w$-version of Kaplansky's theorem, we first introduce and
study a class of modules closely related to the $w$-projective
modules called $w$-split modules (see Section 2). Then we prove, in
Section 3, the Kaplansky's theorem for $w$-projective $w$-modules in
terms of $w$-split modules. More precisely, it is shown that every
$w$-projective $w$-module can be filtered by countably generated and
$w$-split modules (see Theorem \ref{main result}).

Any undefined notions and notations are standard, as in
\cite{R79,GT06,WK16}.

\section{On $w$-split modules}

In this section, we introduce and study $w$-split modules, which can
be used to prove a $w$-version of Kaplansky's theorem on projective
modules.

Before we give the definition of $w$-split modules, we first fix the
following notation. If $M$ is an $R$-module and $s\in R$, then let
$\eta^M_s:M\rightarrow M$ denote the multiplication map $m\mapsto
sm$. It is obvious that $\eta^M_1=\textbf{1}_M$ is the identity map
on $M$. Moreover, if $M, N$ are $R$-modules, then for any $a\in R$
and any $f\in\Hom_R(M,N)$, we have the multiplication $a\cdot
f=f\eta^M_a=\eta^N_af$.

\begin{definition}\quad
\begin{enumerate}
    \item A short exact sequence of $R$-modules
$$0\rightarrow A\stackrel{f}{\rightarrow} B\stackrel{g}{\rightarrow} C\rightarrow 0$$
is said to be \textit{$w$-split} if there exist $J=\langle
d_1,\dots,d_n\rangle\in\GV(R)$ and $h_1,\dots,h_n\in\Hom_R(C,B)$
such that $\eta_{d_k}^C=gh_k$ for all $k=1,\dots,n$.
    \item An $R$-module $M$ is said to be \textit{$w$-split} if
there is a $w$-split short exact sequence of $R$-modules
$$0\rightarrow K\rightarrow P\rightarrow M\rightarrow 0$$ with $P$
projective.
\end{enumerate}
\end{definition}

The first part of the following lemma is exactly \cite[Exercise
1.60]{WK16} (without proof). However, for the purposes of showing
the second part of the lemma, we still give a proof for it.

\begin{lemma}\label{diagram} Consider the following commutative diagram with exact
rows:
$$\xymatrix{
  0 \ar[r] & A \ar[d]_{\alpha} \ar[r]^{f} & B \ar[d]_{\beta} \ar[r]^{g} \ar@{-->}[dl]_{h}& C
  \ar[d]^{\gamma} \ar[r]\ar@{-->}[dl]^{h^\prime} & 0\\
  0 \ar[r] & A^\prime \ar[r]_{f^\prime} & B^\prime \ar[r]_{g^\prime} & C^\prime \ar[r] & 0.}$$
Then there exists a homomorphism $h:B\rightarrow A^\prime$ with
$hf=\alpha$ if and only if there is a homomorphism
$h^\prime:C\rightarrow B^\prime$ with $g^\prime h^\prime=\gamma$. In
this case, the equality $\beta=f^{\prime}h+h^{\prime}g$ holds.
\end{lemma}
\begin{proof}
Assume that there is a homomorphism $h:B\rightarrow A^\prime$ such
that $hf=\alpha$. Then for any $c\in C$, since $g$ is epic, $c=g(b)$
for some $b\in B$. Define $h^\prime:C\rightarrow B^\prime$ by
$$h^\prime(c)=\beta(b)-f^\prime h(b).$$ If $c=g(b)=g(b_1)$ for some
$b_1\in B$, then $b-b_1\in\ker(g)=\im(f)$. Therefore, $b-b_1=f(a)$
with $a\in A$, and so $$\beta(b-b_1)-f^\prime h(b-b_1)=\beta
f(a)-f^\prime hf(a)=\beta f(a)-f^\prime\alpha(a)=0,$$ i.e.,
$\beta(b)-f^\prime h(b)=\beta(b_1)-f^\prime h(b_1)$. Hence,
$h^\prime$ is a well-defined homomorphism with $g^\prime
h^\prime=\gamma$. Conversely, let $h^\prime:C\rightarrow B^\prime$
be a homomorphism with $g^\prime h^\prime=\gamma$. Then for each
$b\in B$, we have $$g^\prime\beta(b)-g^\prime h^\prime g(b)=\gamma
g(b)-\gamma g(b)=0,$$ and so $\beta(b)-h^\prime
g(b)\in\ker(g^\prime)=\im(f^\prime).$ Since $f^\prime$ is monic,
there is a unique $a^\prime\in A^\prime$ with
$f^\prime(a^\prime)=\beta(b)-h^\prime g(b)$. Now, define
$h:B\rightarrow A^\prime$ by $h(b)=a^\prime$. Then it is easy to
check that $h$ is a well-defined homomorphism. For any $a\in A$,
write $a^\prime=hf(a)$. Then $$f^\prime
hf(a)=f^\prime(a^\prime)=\beta f(a)-h^\prime
gf(a)=f^\prime\alpha(a).$$ Thus, it follows that $hf=\alpha$.

Moreover, the second statement follows immediately by the above proof.
\end{proof}

By using the lemma above, it is easy to prove the following
proposition.

\begin{proposition}\label{left w-split} An exact sequence of $R$-modules
$0\rightarrow A\stackrel{f}{\rightarrow} B\stackrel{g}{\rightarrow}
C\rightarrow 0$ is $w$-split if and only if there are $J=\langle
d_1,\dots,d_n\rangle\in\GV(R)$ and $q_1,\dots,q_n\in\Hom_R(B,A)$
such that $\eta_{d_k}^A=q_kf$ for all $k=1,\dots,n$. In this case, for each $k$,
the equality $\eta_{d_k}^B=fq_k+h_kg$ holds, where $h_k$ is as in the definition.
\end{proposition}
\begin{proof}
Assume that the exact sequence $0\rightarrow
A\stackrel{f}{\rightarrow} B\stackrel{g}{\rightarrow} C\rightarrow
0$ is $w$-split. Then there exist $J=\langle
d_1,\dots,d_n\rangle\in\GV(R)$ and $h_1,\dots,h_n\in\Hom_R(C,B)$
such that $\eta_{d_k}^C=gh_k$ for all $k=1,\dots,n$. For each $k$,
let us consider the following commutative diagram with exact rows:
$$\xymatrix{
  0 \ar[r] & A \ar[d]_{\eta^A_{d_k}} \ar[r]^{f} & B \ar[d]^{\eta^B_{d_k}} \ar[r]^{g} \ar@{-->}[dl]_{q_k}& C
  \ar[d]^{\eta^C_{d_k}} \ar[r]\ar[dl]^{h_k} & 0\\
  0 \ar[r] & A \ar[r]_{f} & B\ar[r]_{g} & C \ar[r] & 0.}$$
Thus, by Lemma \ref{diagram}, there is a homomorphism
$q_k:B\rightarrow A$ with $\eta_{d_k}^A=q_kf$ and
$\eta_{d_k}^B=fq_k+h_kg$.

 In a similar way, we can see that the converse is also true.
\end{proof}

The following lemma gives a condition under which a GV-torsion
module is $w$-split. Also, it will be used to characterize $w$-split
modules.

\begin{lemma}\label{GV-torsion and w-split} Let $M$ be a
$\GV$-torsion $R$-module. Then $M$ is $w$-split if and only if there
exists a $J\in\GV(R)$ with $JM=0$.
\end{lemma}
\begin{proof} If $M$ is a $w$-split module, then there is a $w$-split short exact
sequence $0\rightarrow A\rightarrow P\stackrel{g}{\rightarrow}
M\rightarrow 0$ with $P$ a projective module. Then we can pick
$J=\langle d_1,\dots,d_n\rangle\in\GV(R)$ and $h_k:M\rightarrow P$
as in definition. Since $P$ is $\GV$-torsionfree, we must have
$h_k=0$ for all $k=1,\dots,n$. Hence, for each $x\in M$,
$d_kx=\eta^M_{d_k}(x)=gh_k(x)=0$ for all $k$, and so $JM=0$.

Conversely, suppose that there is a $J\in\GV(R)$ with $JM=0$.
Consider a short exact sequence $0\rightarrow A\rightarrow
P\stackrel{g}{\rightarrow} M\rightarrow 0$ of $R$-modules with $P$
projective. Set $J=\langle d_1,\dots,d_n\rangle$ and let
$h_k:M\rightarrow P$ be the zero maps for all $k=1,\dots,n$. Then
for any $x\in M$, we obtain $d_kx=0=gh_k(x)$, and so
$\eta_{d_k}^M=gh_k$ for each $k$. Therefore, $M$ is a $w$-split
module.
\end{proof}

\begin{proposition}\label{characterizations of w-split modules}
The following statements are equivalent for an
$R$-module $M$.
\begin{enumerate}
    \item $M$ is a $w$-split module.
    \item $\Ext^1_R(M,N)$ is $\GV$-torsion for all $R$-modules $N$.
    \item $\Ext^i_R(M,N)$ is $\GV$-torsion for all $R$-modules $N$ and for all integers $i\geq 1$.
    \item For any $R$-epimorphism $g:B\rightarrow C$, the induced
map $$g_*:\Hom_R(M,B)\rightarrow \Hom_R(M,C)$$ has a $\GV$-torsion
cokernel.
    \item For any $R$-epimorphism $g^\prime:F\rightarrow M$ with $F$ projective, the induced
map $$g_*^\prime:\Hom_R(M,F)\rightarrow \Hom_R(M,M)$$ has a
$\GV$-torsion cokernel.
    \item For any $R$-epimorphism $g:B\rightarrow C$ and for each homomorphism $\alpha:M\rightarrow C$,
there exist $J=\langle d_1,\dots,d_n\rangle\in\GV(R)$ and
homomorphisms $h_k:M\rightarrow B$ such that $gh_k=d_k\cdot\alpha$ where
$k=1,\dots,n$.
    \item Every exact sequence $0\rightarrow
A\stackrel{f}\rightarrow B\stackrel{g}\rightarrow M\rightarrow 0$ of
$R$-modules is $w$-split.
    \item There exist elements $\{x_i\}_{i\in I}$ of $M$ and $J=\langle
d_1,\dots,d_n\rangle\in\GV(R)$ such that for all $k=1,\dots,n$,
there are homomorphisms $\{f_{k_i}\in M^*\}_{i\in I}$ satisfying
that for each $x\in M$, almost all $f_{k_i}(x)=0$ and
$d_kx=\sum\limits_i f_{k_i}(x)x_i$.
\end{enumerate}
\end{proposition}
\begin{proof}
$(1)\Rightarrow (2)$ Suppose that $M$ is a $w$-split module and let
$N$ be an $R$-module. Then there exists a $w$-split exact sequence
$0\rightarrow L\stackrel{f}{\rightarrow} F\stackrel{g}{\rightarrow}
M\rightarrow 0$ of $R$-modules with $F$ projective. Since
$$0\rightarrow \Hom_R(M,N)\rightarrow
\Hom_R(F,N)\stackrel{f^*}\longrightarrow \Hom_R(L,N)\rightarrow
\Ext^1_R(M,N)\rightarrow 0$$ is exact, it suffices to prove that
$\cok(f^*)$ is a $\GV$-torsion module. But this is equivalent to
showing that for any $\alpha\in\Hom_R(L,N)$, there is a $J\in\GV(R)$
with $J\alpha\subseteq\im(f^*)$. Now let $\alpha\in\Hom_R(L,N)$ be
arbitrary. Pick $J$ and $q_k:F\rightarrow L$ as in Proposition
\ref{left w-split}, and set $\beta_k=\alpha q_k$, where
$k=1,\dots,n$. Then $f^*(\beta_k)=\beta_k
f=\alpha\eta^L_{d_k}=d_k\cdot \alpha$ for all $k$. Thus, it follows
that $J\alpha\subseteq\im(f^*)$, as desired.

$(2)\Rightarrow (3)$ It follows from standard homological algebra.

$(3)\Rightarrow (4)$ Suppose that (3) holds and let $g:B\rightarrow
C$ be an epimorphism of $R$-modules. Then the sequence
$$\Hom_R(M,B)\stackrel{g_*}\longrightarrow \Hom_R(M,C)\rightarrow
\Ext^1_R(M,\ker(g))$$ is exact. Thus, by (3), $\cok(g_*)$ is
$\GV$-torsion.

$(4)\Rightarrow (5)$ is trivial.

$(5)\Rightarrow (1)$ This follows easily from the definition of
$w$-split modules.

$(4)\Rightarrow (6)$ Assume that (4) holds. Let $g:B\rightarrow C$
be an epimorphism of $R$-modules and $\alpha:M\rightarrow C$ a
homomorphism. Then $\cok(g_*)$ is a $\GV$-torsion module, and so
there exists a $J=\langle d_1,\dots,d_n\rangle\in\GV(R)$ with
$J\alpha\subseteq \im(g_*)$. Thus, we can find
$h_1,\dots,h_n\in\Hom_R(M,B)$ such that for each $k=1,\dots,n$,
$d_k\cdot\alpha=gh_k$, and so (6) holds.

$(6)\Rightarrow (7)$ Apply (6) to the identity map
$\textbf{1}_M:M\rightarrow M$.

$(7)\Rightarrow (8)$ Let $0\rightarrow A\rightarrow
F\stackrel{g}{\rightarrow} M\rightarrow 0$ be an exact sequence of
$R$-modules with $F$ free. Then by (7), it is $w$-split, and so
there exist $J=\langle d_1,\dots,d_n\rangle\in\GV(R)$ and
$h_1,\dots,h_n\in\Hom_R(M,F)$ such that $\eta_{d_k}^M=gh_k$ for all
$k=1,\dots,n$. Now, let $\{e_i\}_{i\in I}$ be a basis of $F$ and set
$x_i=g(e_i)$ for each $i\in I$. Then for any $x\in M$,
$h_k(x)=\sum\limits_i r_{k_i}e_i$, where $r_{k_i}\in R$ and only
finitely many $r_{k_i}\neq 0$. Define $f_{k_i}:M\rightarrow R$ for
all $k=1,\dots,n$ and for all $i\in I$, by $f_{k_i}(x)=r_{k_i}$. It
is clear that all $f_{k_i}\in M^*$ and, for any $x\in M$, we have
$d_k x=\sum\limits_i f_{k_i}(x)x_i$. Hence, (8) holds.

$(8)\Rightarrow (1)$ Assume that (8) holds. Define $F$ to be the
free $R$-module with basis $\{e_i\}_{i\in I}$, and define an $R$-map
$g:F\rightarrow M^\prime$ by $g:e_i\mapsto x_i$, where $M^\prime$ is
the $R$-submodule of $M$ generated by the set $\{x_i\}_{i\in I}$.
Then we obtain an exact sequence $0\rightarrow A\rightarrow
F\stackrel{g}{\rightarrow} M^\prime\rightarrow 0$ with $A=\ker(g)$.
For each $k=1,\dots,n$, define $h_k:M^\prime\rightarrow F$ by
$h_k(y)=\sum\limits_i f_{k_i}(y)e_i$, where $y\in M^\prime$. Since
the sum is finite, $h_k$ is a well-defined homomorphism. Moreover,
for any $y\in M^\prime$, $$gh_k(y)=g\left(\sum\limits_i
f_{k_i}(y)e_i\right)=\sum\limits_i
f_{k_i}(y)x_i=d_ky=\eta^{M^\prime}_{d_k}(y),$$ i.e.,
$\eta_{d_k}^{M^\prime}=gh_k$. Therefore, $M^\prime$ is a $w$-split
module.

Finally, notice that $J(M/M^\prime)=0$ and $M/M^\prime$ is a
GV-torsion $R$-module. Then $M/M^\prime$ is $w$-split by Lemma
\ref{GV-torsion and w-split}. Thus, the equivalence of (1) and (2)
implies that $M$ is also $w$-split.
\end{proof}

As a corollary of Proposition \ref{characterizations of w-split
modules}, we have:

\begin{corollary}\label{corollary of w-split} The following statements hold.
\begin{enumerate}
    \item Every $w$-split module is $w$-projective.
    \item Let $0\rightarrow A\rightarrow B\rightarrow C\rightarrow
    0$ be an exact sequence of $R$-modules with $C$ $w$-split. Then
    $A$ is $w$-split if and only if so is $B$.
\end{enumerate}
\end{corollary}

Next, we will give an example of a $w$-projective module, which is
not $w$-split.

\begin{example} Let $R$ be a two dimensional regular local ring with the
maximal ideal $\frak m$ and set
$$\mbox{$M=\bigoplus\{R/J~|~J\in\GV(R)\}$}.$$ Then $M$ is a $\GV$-torsion
module, and so it is $w$-projective. Now, we say that $M$ is not
$w$-split. If not, then there is a $J_0\in \GV(R)$ with $J_0M=0$ by
Lemma \ref{GV-torsion and w-split}. This means that $J_0$ is
contained in all $J\in\GV(R)$. We claim next that $\frak
m\in\GV(R)$. Indeed, since $R$ is a two dimensional regular local
ring, $\frak m$ is generated by an $R$-sequence of length two, and
so $$\frak m^*=\Hom_R(\frak m, R)\cong \frak m^{-1}=R$$ by \cite[p.
102, Exercise 1]{K74}, where $\frak m^{-1}=\{x\in Q~|~\frak
mx\subseteq R\}$ and $Q$ is the quotient field of $R$. Therefore,
$\frak m\in\GV(R)$, and so is $\frak m^n$ for any integer $n\geq 1$.
Thus, it follows that
$$\mbox{$J_0\subseteq \bigcap\limits_{n=1}^{\infty}{\frak m}^n$=0,}$$
whence $J_0=0$. However, this means that $R$ as a module over itself
is both GV-torsion and GV-torsionfree, hence $R=0$, which is a
contradiction.
\end{example}

We close this section with a short discussion of when a
$w$-projective module is $w$-split.

\begin{proposition}\label{w-projective w-module is w-split} Every
$w$-projective $w$-module is $w$-split.
\end{proposition}
\begin{proof}  Let $M$ be a $w$-projective $w$-module over $R$ and
let $g:P\rightarrow M$ be an $R$-epimorphism with $P$ projective. It
follows that $0\rightarrow K\rightarrow P\stackrel{g}{\rightarrow}
M\rightarrow 0$ is exact, where $K=\ker(g)$. Then we have the
following exact sequence
$$\Hom_R(M,P)\stackrel{g_*}{\longrightarrow} \Hom_R(M,M)\rightarrow
\Ext^1_R(M,K)$$ Since $M$ is $\GV$-torsionfree, $K$ is a torsionfree
$w$-module, and so $\Ext^1_R(M,K)$ is $\GV$-torsion by the
$w$-projectivity of $M$. Thus, $\cok(g_*)$ is also $\GV$-torsion,
whence $M$ is a $w$-split module by Proposition
\ref{characterizations of w-split modules}.
\end{proof}

\begin{proposition}\label{when a w-projective module is w-split} Let
$M$ be a $\GV$-torsionfree $w$-projective $R$-module. Then $M$ is
$w$-split if and only if there is a $J\in\GV(R)$ with $JM_w\subseteq
M$.
\end{proposition}
\begin{proof} Write $T=M_w/M$. Then $T$ is a $\GV$-torsion module and
the sequence $0\rightarrow M\stackrel{\mu}\rightarrow
M_w\stackrel{\pi}\rightarrow T\rightarrow 0$ is exact, where $\mu$
is the inclusion map and $\pi$ is the natural map. Also, note that
$L(M)=M_w=L(M_w)$. Thus it follows, from the definition of
$w$-projective modules and the $w$-projectivity of $M$, that $M_w$
is also $w$-projective. Hence, it follows from Proposition
\ref{w-projective w-module is w-split} that $M_w$ is $w$-split.

If there is a $J\in\GV(R)$ with $JM_w\subseteq M$, then $JT=0$, and
so $T$ is $w$-split by Lemma \ref{GV-torsion and w-split}.
Therefore, Corollary \ref{corollary of w-split}(2) says that $M$ is
$w$-split.

To prove the converse, it suffices by Lemma \ref{GV-torsion and
w-split} to show that $T$ is a $w$-split module. This in turn is
equivalent to prove that for each $R$-module $N$, $\Ext^1_R(T,N)$ is
$\GV$-torsion. Let $N$ be an arbitrary $R$-module. Thus, consider
the following exact sequence
$$\Hom_R(M_w,N)\stackrel{\mu^*}\longrightarrow
\Hom_R(M,N)\rightarrow \Ext^1_R(T,N)\rightarrow \Ext^1_R(M_w,N).$$
Then since $M_w$ is $w$-split, $\Ext^1_R(M_w,N)$ is $\GV$-torsion.
Hence, to complete the proof, we need only to show that
$\cok(\mu^*)$ is $\GV$-torsion as well. As $M$ is a $w$-split
module, there is a $w$-split short exact sequence $0\rightarrow
K\rightarrow P\stackrel{g}{\rightarrow} M\rightarrow 0$ with $P$
projective, hence there exists $J=\langle
d_1,\dots,d_n\rangle\in\GV(R)$ and homomorphisms
$h_1,\dots,h_n:M\rightarrow P$ such that $\eta_{d_k}^M=gh_k$, for
all $k=1,\dots,n$. Now, let $\alpha\in\Hom_R(M,N)$ and consider the
following commutative diagram with exact rows
$$\xymatrix{0 \ar[r] & K \ar[d]\ar[r] & P\ar[d]^{\beta}\ar[r]^{g} &
M \ar[d]^{\alpha}\ar[r] & 0\\ 0 \ar[r] & A\ar[r] &
Q\ar[r]_{g^\prime} & N\ar[r] & 0,}$$ where $P$ and $Q$ are
projective modules. Since $P$ is a $w$-module, each $h_k$ can be
extended to a homomorphism $h^\prime_k:M_w\rightarrow P$. For any
$k$, set $f_k=g^\prime\beta h^\prime_k$. Then it is easy to check
that $\mu^*(f_k)=d_k\cdot\alpha$ for all $k$. Hence, it follows that
$J\alpha\subseteq \im(\mu^*)$, i.e., $\cok(\mu^*)$ is $\GV$-torsion.
\end{proof}

\section{Kaplansky's theorem for $w$-projective $w$-modules}

To begin with, we fix for all the section the following notation.
Let $$0\rightarrow P\xrightarrow{f} F\xrightarrow{g} M\rightarrow 0
\eqno{(\xi)}$$ be a $w$-split exact sequence of $R$-modules and
$F_1$ a submodule of $F$. Then pick $J\in\GV(R)$ and
$h_1,\dots,h_n\in\Hom_R(M,F)$ as in the definition in Section 2 and
write $$\mbox{$g_1=g|_{F_1}$, $M_1=\im(g_1)$,
$P_1=f^{-1}(\ker(g_1))$ and $f_1=f|_{P_1}$}.$$ Thus, the sequence
$$0\rightarrow P_1\xrightarrow{f_1} F_1\xrightarrow{g_1}
M_1\rightarrow 0 \eqno{(\xi_1)}$$ is also exact. Moreover, if for
each $k=1,\dots,n$, $h_k(M_1)\subseteq F_1$, then $(\xi_1)$ is also
$w$-split. In this case ($h_k(M_1)\subseteq F_1$ for all
$k=1,\dots,n$), we call $(\xi_1)$ a \textit{$w$-split exact sequence
induced by $(\xi)$ with respect to the submodule $F_1$}.

Let $F$ be as in $(\xi)$ with a direct sum decomposition
$F=\bigoplus\limits_{i\in I}F_i$ of projective submodules, where $I$
is an index set. For each subset $H$ of $I$, if $H=\emptyset$, then
write $$\mbox{$F(H)=0$, $P(H)=0$ and $M(H)=0$;}$$ otherwise, write
$$\mbox{$F(H)=\bigoplus\limits_{j\in H}F_j$, $g_H=g|_{F(H)}$,
$M(H)=\im(g_H)$,}$$ $$\mbox{$P(H)=f^{-1}(\ker(g_H))$ and
$f_H=f|_{P(H)}$.}$$ It is obvious that $F(I)=F$, $M(I)=M$ and
$P(I)=P$, and that if $H_1,H_2$ are subsets of $I$ with
$H_1\subseteq H_2$, then $F(H_1)$ is a direct summand of $F(H_2)$,
$P(H_1)\subseteq P(H_2)$, and $M(H_1)\subseteq M(H_2)$. Now, $M(H)$
is said to be a \textit{$w$-split module induced by $(\xi)$ with
respect to the subset $H$} if the sequence $$0\rightarrow
P(H)\xrightarrow{f_H} F(H)\xrightarrow{g_H} M(H)\rightarrow 0
\eqno{(\xi_H)}$$ is a $w$-split exact sequence induced by $(\xi)$
with respect to the projective submodule $F(H)$.

\begin{lemma}\label{cup of w-split modules}
Let $F$ be the module as in the previous paragraph. Suppose that
$$\mathcal{S}_1=\{H_s~|~H_s\subseteq I\}$$ is a set, totally ordered
by inclusion, satisfying the sequence $$0\rightarrow
P(H_s)\xrightarrow{f_{H_s}} F(H_s)\xrightarrow{g_{H_s}}
M(H_s)\rightarrow 0\eqno{(\xi_{H_s})}$$ is a $w$-split exact
sequence induced by $(\xi)$ with respect to the submodule $F(H_s)$
for any $H_s\in\mathcal{S}_1$. Set
$H=\mathop\bigcup\limits_{H_s\in\mathcal{S}_1}H_s$. Then the
following statements hold.
\begin{enumerate}
\item $\mathop\bigcup\limits_{H_s\in\mathcal{S}_1}F(H_s)=F(H)$
and hence it is a projective module. Moreover,
$\mathop\bigcup\limits_{H_s\in\mathcal{S}_1}P(H_s)=P(H)$ and
$\mathop\bigcup\limits_{H_s\in\mathcal{S}_1}M(H_s)=M(H)$.
\item The sequence
$$0\rightarrow \mathop\bigcup\limits_{H_s\in\mathcal{S}_1}P(H_s)
\xrightarrow{\quad}\mathop\bigcup\limits_{H_s\in\mathcal{S}_1}F(H_s)
\xrightarrow{\quad}\mathop\bigcup\limits_{H_s\in\mathcal{S}_1}M(H_s)
\rightarrow 0
\eqno{\left(\xi_{\mathop\cup_{\mathcal{S}_1}}\right)}$$ is a
$w$-split exact sequence induced by $(\xi)$ with respect to the
submodule $\mathop\bigcup\limits_{H_s\in\mathcal{S}_1}F(H_s)$. Hence
$\mathop\bigcup\limits_{H_s\in\mathcal{S}_1}M(H_s)$ is a $w$-split
module induced by $(\xi)$ with respect to the subset $H$.
\end{enumerate}
\end{lemma}
\begin{proof}
(1) For each $y\in F(H)$, write $y=\sum\limits_{j\in H}y_j$, where
$y_j\neq 0$ for only a finite number of indices $j$. Since
$\mathcal{S}_1$ is totally ordered, we can choose some
$H_{s_0}\in\mathcal{S}_1$ such that $y=\sum\limits_{i\in
H_{s_0}}y_i\in F(H_{s_0})$, and so $F(H)\subseteq
\bigcup\limits_{H_s\in\mathcal{S}_1}F(H_s)$. The other inclusion is
clear.

It is obvious that
$\bigcup\limits_{H_s\in\mathcal{S}_1}P(H_s)\subseteq P(H)$. For the
other inclusion, let $x\in P(H)$. Then we have
$$f(x)\in\ker(g_H)=\ker(g)\bigcap F(H).$$ Since
$F(H)=\bigcup\limits_{H_s\in\mathcal{S}_1}F(H_s)$, there is some
$H_{s_0}\in\mathcal{S}_1$ such that $f(x)\in F(H_{s_0})$, and so
$f(x)\in\ker(g)\bigcap F(H_{s_0})=\ker(g_{H_{s_0}})$. Thus, $x\in
f^{-1}(\ker(g_{H_{s_0}}))=P(H_{s_0})$, whence $P(H)\subseteq
\mathop\bigcup\limits_{H_s\in\mathcal{S}_1}P(H_s)$.

Let $x\in M(H)$. Then $x=g(y)$ for some $y\in F(H)$. There is some
$H_{s_0}\in\mathcal{S}_1$ such that $y\in F(H_{s_0})$. Therefore,
$x\in g\left(F(H_{s_0})\right)=M(H_{s_0})$, whence $M(H)\subseteq
\mathop\bigcup\limits_{H_s\in\mathcal{S}_1}M(H_s)$. The other
inclusion is obvious.

(2) Pick $h_1,\dots,h_n\in\Hom_R(M,F)$ as in the definition of
$w$-split exact sequences. Then we need only show that
$$h_k\left(\bigcup\limits_{H_s\in\mathcal{S}_1}M(H_s)\right)\subseteq
\bigcup\limits_{H_s\in\mathcal{S}_1}F(H_s)$$ for all $k=1,\dots,n$.
Now let $x\in \bigcup\limits_{H_s\in\mathcal{S}_1}M(H_s)$ be
arbitrary. Then $x\in M(H_{s_0})$ for some $H_{s_0}\in
\mathcal{S}_1$. Since $\left(\xi_{H_{s_{0}}}\right)$ is a $w$-split
exact sequence induced by $(\xi)$ with respect to the submodule
$F(H_{s_0})$, we have $h_k\left(M(H_{s_0})\right)\subseteq
F(H_{s_0})$ for any $k$, and so $h_k(x)\in F(H_{s_0})\subseteq
\bigcup\limits_{H_s\in\mathcal{S}_1}F(H_s)$, as desired.

The second assertion of (2) is clear.
\end{proof}

Recall that an ideal $I$ of $R$ is said to be a \textit{$w$-ideal}
if it is a $w$-module as an $R$-module, i.e., $I=I_w$. Let
$\wmax(R)$ denote the set of $w$-ideals of $R$ maximal among proper
integral $w$-ideals of $R$ and we call $\frak m\in\wmax(R)$ a
maximal $w$-ideal of $R$. Then every proper $w$-ideal is contained
in a maximal $w$-ideal and every maximal $w$-ideal is a prime ideal.

Let $M$ be an $R$-module and $S$ a multiplicatively closed subset of
$R$. Then we denote as usual by $M_S$ the localization of $M$ at
$S$. In particular, if $\frak p$ is a prime ideal of $R$ and
$S=R\backslash\frak p$, then $M_S$ is denoted by $M_\frak p$. Let
$f:M\rightarrow N$ be a homomorphism of $R$-modules. For $s\in S$
and $x\in M$, define $$f_S\left(\frac{x}{s}\right)=\frac{f(x)}{s}.$$
Then $f_S:M_S\rightarrow N_S$ is a well-defined $R_S$-homomorphism.

Recall from \cite{WK16} that an $R$-homomorphism $f:M\rightarrow N$
is called a \textit{$w$-isomorphism} if $f_{\frak m}:M_{\frak
m}\rightarrow N_{\frak m}$ is an isomorphism over $R_\frak m$ for
any $\frak m\in\wmax(R)$. Since an $R$-module $M$ is $\GV$-torsion
if and only if $M_\frak m=0$ for any $\frak m\in\wmax(R)$ (see
\cite[Theorem 6.2.15]{WK16}), it is easy to see that a homomorphism
$f:M\rightarrow N$ is a $w$-isomorphism if and only if both
$\ker(f)$ and $\cok(f)$ are $\GV$-torsion.

Now, we call an $R$-module $M$ a \textit{$w$-countably generated
module} if there is a $w$-isomorphism $f:M_0\rightarrow M$ with
$M_0$ a countably generated $R$-module. It is easily seen that $M$
is $w$-countably generated if and only if there exists a submodule
$N$ of $M$ such that for any $\frak m\in\wmax(R)$, $N_\frak
m=M_\frak m$.

\begin{lemma}\label{key lemma}
Let $F$ be as in $(\xi)$. Suppose that $F$ is a $w$-module over $R$
with a direct sum decomposition $F=\bigoplus\limits_{i\in I}F_i$ of
countably generated submodules. For each subset $H$ of $I$, let
$F(H)$, $P(H)$ and $M(H)$ be as before. If $H$ is a proper subset of
$I$ satisfying $(\xi_H)$ is a $w$-split exact sequence induced by
$(\xi)$ with respect to the submodule $F(H)$, then the following
statements hold.
\begin{enumerate}
\item There is a subset $H_1$ of $I$ properly containing $H$ such that
$$0\rightarrow P(H_1)\xrightarrow{f_{H_1}}F(H_1)\xrightarrow{g_{H_1}}
M(H_1)\rightarrow 0\eqno{(\xi_{H_1})}$$ is a $w$-split exact sequence
induced by $(\xi)$ with respect to the submodule $F(H_1)$.
\item $C:=M(H_1)/M(H)$ is a countably generated module.
\item If $M$ is $\GV$-torsionfree, then $C$ is $w$-isomorphic to
$D:=M(H_1)_w/M(H)_w$. In this case, $D$ is a $w$-countably generated module.
\item If $M$ is $\GV$-torsionfree and if each $F_i$ is projective,
then $M(H)$ and $M(H_1)$ are $w$-split modules induced by $(\xi)$
with respect to the subsets $H$ and $H_1$, respectively,
and $C$ is $w$-split. In this case, $D$ is a $w$-countably generated
and $w$-projective module.
\end{enumerate}
\end{lemma}
\begin{proof}
(1) Since $(\xi)$ is a $w$-split exact sequence, we can pick
$J=\langle d_1,\dots,d_n\rangle\in\GV(R)$, $h_k:M\rightarrow F$ and
$q_k:F\rightarrow P$ as in Proposition \ref{left w-split}, where
$k=1,\dots,n$. Then for any $j\in I$, both $fq_k(F_j)$ and
$h_kg(F_j)$ are countably generated modules, and for each $x\in
F_j$, we have $d_kx=fq_k(x)+h_kg(x)$. Therefore, $d_kF_j\subseteq
fq_k(F_j)+h_kg(F_j)$ for all $k$.

Choose an $i_0\in I\backslash H$. Then there exists a countable
subset $I_1$ of $I$ such that $$d_kF_{i_0}\subseteq
fq_k(F_{i_0})+h_kg(F_{i_0})\subseteq\bigoplus\limits_{i\in I_1}F_i$$
for all $k$. Note that $\bigoplus\limits_{i\in I_1}F_i$ is countably
generated as each $F_i$ is countably generated and $I_1$ is a
countable set. Thus, we can find another countable subset $I_2$ of
$I$ containing $I_1$ with $$d_k\left(\bigoplus\limits_{j\in
I_1}F_j\right)\subseteq fq_k\left(\bigoplus\limits_{j\in
I_1}F_j\right)+h_kg\left(\bigoplus\limits_{j\in
I_1}F_j\right)\subseteq\bigoplus\limits_{i\in I_2}F_i$$ for all $k$.
Continuing, we obtain countable subsets
$I_0=\{i_0\},I_1,I_2,\dots,I_s,\dots$ satisfying
$$d_k\left(\bigoplus\limits_{j\in I_s}F_j\right)\subseteq
fq_k\left(\bigoplus\limits_{j\in
I_s}F_j\right)+h_kg\left(\bigoplus\limits_{j\in
I_s}F_j\right)\subseteq\bigoplus\limits_{i\in
I_{s+1}}F_i\eqno{(\dag)}$$ for all $k$. Hence,
$$J\left(\bigoplus\limits_{j\in I_{s}}F_j\right)\subseteq
\bigoplus\limits_{i\in I_{s+1}}F_i.$$ But since $J_w=R$ (cf.
\cite[Proposition 3.5]{YWZC11}) and both $\bigoplus\limits_{j\in
I_{s}}F_j$ and $\bigoplus\limits_{i\in I_{s+1}}F_i$ are $w$-modules
(cf. \cite[Proposition 2.3]{YWZC11}), it follows easily from
\cite[Theorem 6.2.2]{WK16} that $\bigoplus\limits_{j\in
I_{s}}F_j\subseteq\bigoplus\limits_{i\in I_{s+1}}F_i$.

Set $L=\bigcup\limits_{s=0}^{\infty}I_s$. Then it is a countable
set, and so $L_1:=L\backslash H$ is countable too. Write
$H_1=H\bigcup L$. Then $H_1=H\bigcup L_1$. Thus
$V:=\bigoplus\limits_{i\in L_1}F_i$ is countably generated and
$F(H_1)=F(H)\bigoplus V$. For each $j\in H_1$, if $j\in H$, then
since $(\xi_H)$ is a $w$-split exact sequence induced by $(\xi)$
with respect to the submodule $F(H)$, we obtain
$$h_kg(F_j)\subseteq h_k(M(H))\subseteq F(H)\subseteq F(H_1).$$
Otherwise, $j\in L_1$, and so $j\in I_s$ for some $s$. Hence, by
$(\dag)$, we also have $h_kg(F_j)\subseteq F(H_1)$. So it follows
that $h_k(M(H_1))\subseteq F(H_1)$, whence (1) holds.

(2) Consider the following commutative diagram with exact rows and columns
$$\xymatrix{
& 0\ar[d] & 0\ar[d] & 0\ar[d] &\\
0\ar[r] & P(H)\ar[r]\ar[d] & P(H_1)\ar[r]\ar[d] & A\ar[r]\ar[d] & 0\\
0\ar[r] & F(H)\ar[r]\ar[d] & F(H_1)\ar[r]\ar[d] & V\ar[r]\ar[d]^{g^\prime} & 0\\
0\ar[r] & M(H)\ar[r]\ar[d] & M(H_1)\ar[r]\ar[d] & C\ar[r]\ar[d] & 0\\
& 0 & 0 & 0 &
}$$ where $V$ is as in the proof of (1).
Then since $V$ is countably generated, so is $C$.

(3) Consider the following commutative diagram having exact rows.
$$\xymatrix{
0\ar[r] & M(H)\ar[r]\ar[d] & M(H_1)\ar[r]\ar[d] & C\ar[r]\ar[d]^h & 0\\
0\ar[r] & M(H)_w\ar[r] & M(H_1)_w\ar[r] & D\ar[r] & 0
}$$ Then the Snake Lemma implies that the sequence
$$0\rightarrow \ker(h)\rightarrow M(H)_w/M(H)\rightarrow M(H_1)_w/M(H_1)
\rightarrow \cok(h)\rightarrow 0$$ is exact. Therefore, both
$\ker(h)$ and $\cok(h)$ are $\GV$-torsion, i.e., $h$ is a
$w$-isomorphism. Hence, by (2), $D$ is $w$-countably generated.

(4) If each $F_i$ is a projective module, then it is clear that
$M(H)$ and $M(H_1)$ are $w$-split modules induced by $(\xi)$ with
respect to the subsets $H$ and $H_1$, respectively. Let $J=\langle
d_1,\dots,d_n\rangle$ and $h_k$ ($k=1,\dots,n$) as in the proof of
(1). Then to show that $C$ is $w$-split, let us consider, for each
$k$, the following diagram having exact rows
$$\xymatrix{
  0 \ar[r] & M(H) \ar[d]_{h^{(k)}} \ar[r] & M(H_1) \ar[d]^{h^{(k)}_1}
  \ar[r] & C \ar@{-->}[d]^{\alpha_k} \ar[r] & 0 \\
  0 \ar[r] & F(H) \ar[r] & F(H_1) \ar[r] & V \ar[r] & 0   }$$ where
  $h^{(k)}=h_k|_{M(H)}$, $h^{(k)}_1=h_k|_{M(H_1)}$, and $V$ is as in
  the proof of (1). Clearly, the left square commutes, and so there
  is a homomorphism $\alpha_k:C\rightarrow V$ such that the right
  square commutes as well. Let $g^\prime$ be as in the proof of (2).
  Then it is not difficult to see that $g^\prime\alpha_k=\eta^C_{d_k}$ for all $k$.
  Thus, $0\rightarrow A\rightarrow V\stackrel{g^\prime}\longrightarrow C\rightarrow
  0$ is a $w$-split exact sequence with $V$ projective, and consequently $C$ is $w$-split.
\end{proof}

Let $\alpha$ be an ordinal and
$\mathcal{A}=(A_\lambda~|~\lambda\leq\alpha)$ a sequence of modules.
Then $\mathcal{A}$ is called a \textit{continuous chain of modules}
(see \cite{GT06}) if $A_0=0$, $A_{\lambda}\subseteq A_{\lambda+1}$
for all $\lambda<\alpha$ and
$A_{\lambda}=\bigcup\limits_{\mu<\lambda}A_{\mu}$ for all limit
ordinals $\lambda\leq\alpha$.

Let $M$ be a $w$-split $R$-module. Then there exists a $w$-split
exact sequence $(\xi)$ with $F$ projective. Write
$F=\bigoplus\limits_{i\in I}F_i$, where each $F_i$ is a countably
generated projective module. Now, a submodule $N$ of $M$ is said to
be \textit{filtered} by countably generated $w$-split modules if,
for some ordinal $\alpha$, there exist a subset $H$ of $I$ and a
sequence $$\mathcal{H}=(H_\lambda~|~\lambda\leq\alpha)$$ of subsets
of $H$ such that $H_0=\emptyset$, $H_\alpha=H$ and
\begin{enumerate}
    \item[(i)] the sequence of modules $$\mathcal{N}=
    \left(N_\lambda:=M(H_\lambda)~|~\lambda\leq\alpha\right)\eqno{(\ddag)}$$
    is a continuous chain of $N$ with $N_\alpha=N$;
    \item[(ii)] each $N_\lambda$ is a $w$-split module induced by
    $(\xi)$ with respect to the subset $H_\lambda$;
    \item[(iii)] for each $\lambda<\alpha$, $N_{\lambda+1}/N_{\lambda}$
    is a countably generated $w$-split module.
\end{enumerate}
The chain $\mathcal{N}$ is called a \textit{countably generated
$w$-split filtration} of $N$.

In this case, if $B$ is also a submodule of $M$ and if, for some
ordinal $\beta\leq \alpha$, $B$ is filtered by countably generated
$w$-split modules having a continuous chain
$$\mathcal{B}=(B_\mu:=M(H_\mu)~|~\mu\leq\beta)$$ of submodules such
that $B_\mu=N_\mu$ for each $\mu\leq\beta$, then we call $N$ a
\textit{filtered extension} of $B$ by countably generated $w$-split
modules.

\begin{lemma}\label{cup of filtered modules}
Let, as in the $(\xi)$ above, $M$ be a $w$-split module and $F$ a
projective module with a direct sum decomposition
$F=\bigoplus\limits_{i\in I}F_i$ of countably generated projective
submodules. Suppose that $\{H_j\}_{j\in \Gamma}$ is a totally
ordered family of subsets of $I$ satisfying that for each
$j\in\Gamma$, $A_j:=M(H_j)$ can be filtered by countably generated
$w$-split modules, and that for $j,k\in \Gamma$, if $H_j\subseteq
H_k$, then $A_k$ is a filtered extension of $A_j$ by countably
generated $w$-split modules. Then $N:=\bigcup\limits_{j\in
\Gamma}A_j$ can also be filtered by countably generated $w$-split
modules and for each $j\in \Gamma$, it is a filtered extension of
$A_j$.
\end{lemma}
\begin{proof}
Notice that, if $N=A_j$ for some $j\in\Gamma$, then we have nothing
to prove. So we assume that for any $j\in\Gamma$, $N\neq A_j$. For
all $j\in \Gamma$, since $A_j$ is filtered by countably generated
$w$-split modules, there are an ordinal $\alpha_j$ and $H_j$ a
subset of $I$ such that
$$\left(\left.A^{(j)}_{\lambda_j}:=M\left(H^{(j)}_{\lambda_j}\right)
~\right|~\lambda_j\leq\alpha_j\right)$$ is a countably generated
$w$-split filtration of $A_j$, where
$\left(H^{(j)}_{\lambda_j}~|~\lambda_j\leq\alpha_j\right)$ is a
sequence of subsets of $H_j$ with $H^{(j)}_0=\emptyset$ and
$H^{(j)}_{\alpha_j}=H_j$.

Set $\alpha=\mathop\bigcup\limits_{j\in\Gamma}\alpha_j$. Then we
claim that $\alpha\notin\{\alpha_j\}_{j\in\Gamma}$, and so it is a
limit ordinal. Otherwise, $\alpha=\alpha_j$ for some $j\in\Gamma$.
Since $N\neq A_j$, there exists a $k\in\Gamma$ with $A_k\nsubseteq
A_j$, whence $A_j\subset A_k$. Hence, it follows, from the
definition of filtered extensions above, that
$$\mbox{$\alpha_j\leq\alpha_k(\leq\alpha=\alpha_j)$ and
$A_j=A^{(j)}_{\alpha_j}=A^{(k)}_{\alpha_j}=A^{(k)}_{\alpha_k}=A_k$,}$$
which is a contradiction.

Set $H=\mathop\bigcup\limits_{j\in\Gamma}H_j$. Then we construct a
sequence $$\mathcal{H}=(H_\lambda~|~\lambda\leq\alpha)$$ of subsets
of $H$ as follows.

\begin{enumerate}
\item Write $H_0=\emptyset$ and $H_\alpha=H$.
\item Let $0<\lambda<\alpha$. Then we must have some $j\in\Gamma$ with
$\lambda<\alpha_j$. Define $H_\lambda=H^{(j)}_\lambda$.
\end{enumerate}

Next, we prove that the sequence of submodules
$$\mathcal{N}=\left(N_\lambda:=M(H_\lambda)~|~\lambda\leq\alpha\right)$$
is a countably generated $w$-split filtration of $N$.

\begin{enumerate}
\item[(i)] The sequence $\mathcal{N}$ is a a continuous chain of $N$ with
$N_\alpha=N$.
        \begin{enumerate}
        \item $N_0=M(H_0)=0$ is clear.
        \item $N_\alpha=N$.

        By Lemma \ref{cup of w-split modules}, we have
        $$N_\alpha=M(H_\alpha)=M(H)=\bigcup\limits_{j\in\Gamma}M(H_j)
        =\bigcup\limits_{j\in\Gamma}A_j=N.$$
        \item $N_\lambda\subseteq N_{\lambda +1}$ for all $\lambda<\alpha$.

        If $\lambda<\alpha$, then $\lambda+1<\alpha$ as $\alpha$ is
        a limit ordinal. Consequently, $\lambda<\lambda+1<\alpha_j$
        for some $j\in\Gamma$, and hence $$N_\lambda=M(H_\lambda)
        =M\left(H^{(j)}_{\lambda}\right)\subseteq M\left(H^{(j)}_{\lambda+1}\right)
        =M(H_{\lambda+1})=N_{\lambda+1}.$$
        \item $N_\lambda=\bigcup\limits_{\beta<\lambda}N_\beta$
        for all limit ordinals $\lambda\leq \alpha$.

        Note that we need only consider the case $\lambda=\alpha$.
        For each $j\in\Gamma$, we have some $k$ with $\alpha_j<\alpha_k$,
        and so $A_j\subset A_k$. It follows that
        $$N_{\alpha_j}=M\left(H_{\alpha_j}\right)=M\left(H^{(k)}_{\alpha_j}\right)
        =M\left(H^{(j)}_{\alpha_j}\right)=M\left(H_j\right)=A_j.$$
        This means that $\bigcup\limits_{\beta<\alpha}N_\beta$ contains all $A_j$'s,
        whence $\bigcup\limits_{\beta<\alpha}N_\beta=N=N_\alpha$.
        \end{enumerate}
\item[(ii)] Each $N_\lambda$ is a $w$-split module induced
by $(\xi)$ with respect to the subset $H_\lambda$.

Clearly, this is true for any $\lambda<\alpha$. For the case
$\lambda=\alpha$, it is proved in Lemma \ref{cup of w-split
modules}(3).
\item[(iii)] For each $\lambda<\alpha$, $N_{\lambda+1}/N_{\lambda}$
is a countably generated $w$-split
module.

The proof is similar to that of (c) of (i).
\end{enumerate}

Thus, we see that $N$ is also filtered by countably generated
$w$-split modules. While the second assertion follows immediately
from the constructions of $\mathcal{H}$ and $\mathcal{N}$.
\end{proof}

\begin{remark} For the construction of $\mathcal{H}$ in the proof of
Lemma \ref{cup of filtered modules}, if $0<\lambda<\alpha$ and if
there are $j_1,j_2\in\Gamma$ such that $\lambda<\alpha_{j_1}$ and
$\lambda<\alpha_{j_2}$, then either $H_{j_1}\subseteq H_{j_2}$ or
$H_{j_2}\subseteq H_{j_1}$ as $\{H_j\}_{j\in \Gamma}$ is totally
ordered, and so either $A_{j_1}\subseteq A_{j_2}$ or
$A_{j_2}\subseteq A_{j_1}$. However, in both cases, from the
definition of filtered extensions, it is easy to see that
$$M\left(H^{(j_1)}_{\lambda}\right)=M\left(H^{(j_2)}_{\lambda}\right).$$
Thus, $H_\lambda$ can be defined to be any of $H^{(j_1)}_{\lambda}$
and $H^{(j_2)}_{\lambda}$.
\end{remark}

Similarly, if $M$ is a $w$-projective $w$-module over $R$ and if for
some ordinal $\alpha$, there is a continuous chain
$$\mathcal{M}=(M^\prime_\lambda~|~\lambda\leq\alpha)$$ of
$w$-projective $w$-submodules such that $M^\prime_\alpha=M$ and
$M^\prime_{\lambda+1}/M^\prime_{\lambda}$ is a $w$-countably
generated $w$-projective module for each $\lambda<\alpha$, then $M$
is said to be \textit{filtered} by $w$-countably generated
$w$-projective modules.

Now, we can prove the Kaplansky's theorem for $w$-projective
$w$-modules.

\begin{theorem}\label{main result} Let $M$ be a $w$-projective
$w$-module. Then
\begin{enumerate}
    \item $M$ can be filtered by countably generated $w$-split modules.
    \item $M$ can be filtered by $w$-countably generated $w$-projective modules.
\end{enumerate}
\end{theorem}
\begin{proof} (1) Since $M$ is a $w$-projective $w$-module, it is
$w$-split by Proposition \ref{w-projective w-module is w-split}. Let
as before $(\xi)$ be the $w$-split exact sequence with
$F=\bigoplus\limits_{i\in I}F_i$ a projective module, where each
$F_i$ is a countably generated projective module. With the same
notation as in Lemma \ref{cup of w-split modules}, let $\mathcal{S}$
be a collection of subsets $H$ of $I$ satisfying:
\begin{enumerate}
    \item[(a)] $M(H)$ is a $w$-split module induced by $(\xi)$
    with respect to the subset $H$;
    \item[(b)] $M(H)$ can be filtered by countably generated
    $w$-split modules.
\end{enumerate}
Clearly, $\mathcal{S}$ is non-empty as it contains $\emptyset$.
Define a partial order $\leqslant$ on $\mathcal{S}$ by $H_1\leqslant
H_2\Leftrightarrow H_1\subseteq H_2$ and $M(H_2)$ is a filtered
extension of $M(H_1)$ by countably generated $w$-split modules. Let
$\mathcal{S}_1$$=\{H_s\}$ be a totally ordered subset of
$\mathcal{S}$ and $H=\bigcup\limits_{H_s\in\mathcal{S}_1}H_s$. Then
Lemma \ref{cup of w-split modules} says that
$M(H)=\bigcup\limits_{H_s\in\mathcal{S}_1}M(H_s)$ is a $w$-split
module induced by $(\xi)$ with respect to the subset $H$.

Now, we claim that the second condition of Lemma \ref{cup of
filtered modules} is also satisfied. Indeed, as $\mathcal{S}_1$ is
totally ordered by $\leqslant$, it is also totally ordered by
inclusion. Moreover, by the choice of $\mathcal{S}$, for each
$H_s\in\mathcal{S}_1\subseteq \mathcal{S}$, $M(H_s)$ can be filtered
by countably generated $w$-split modules. Also, if
$H_{s_1},H_{s_2}\in\mathcal{S}_1$ with $H_{s_1}\subseteq H_{s_2}$,
then either $H_{s_1}\leqslant H_{s_2}$ or $H_{s_2}\leqslant
H_{s_1}$. In the first case, the definition of $\leqslant$ implies
that $M(H_{s_2})$ is a filtered extension of $M(H_{s_1})$ by
countably generated $w$-split modules. In the second case,
$H_{s_2}\leqslant H_{s_1}$, we have $H_{s_2}\subseteq H_{s_1}$,
whence $H_{s_1}=H_{s_2}$ and $M(H_{s_1})=M(H_{s_2})$. But it is
obvious that $M(H_{s_2})$ is a filtered extension of itself
($M(H_{s_1})$) by countably generated $w$-split modules.

Thus, it follows from Lemma \ref{cup of filtered modules} that
$M(H)$ can be filtered by countably generated $w$-split modules and
that for each $H_s\in \mathcal{S}_1$, $M(H)$ is a filtered extension
of $M(H_s)$ by countably generated $w$-split modules. Thus, $H\in
\mathcal{S}$ and it is a upper bound of $\mathcal{S}_1$. By Zorn's
Lemma, $\mathcal{S}$ has a maximal element, say, $H$.

If $H\neq I$, then by Lemma \ref{key lemma}, there is a subset $H_1$
of $I$ properly containing $H$ such that $M(H_1)$ is a $w$-split
module induced by $(\xi)$ with respect to the subset $H_1$  and that
$C=M(H_1)/M(H)$ is a countably generated $w$-split module.
Therefore, if $\left(M_\lambda(H)~|~\lambda\leq\alpha\right)$ is a
countably generated $w$-split filtration of $M(H)$, then by setting
$M_{\alpha+1}(H_1)=M(H_1)$ and $M_\lambda(H_1)=M_\lambda(H)$ for
$\lambda\leq\alpha$, we see that
$\left(M_\lambda(H_1)~|~\lambda\leq\alpha+1\right)$ is a countably
generated $w$-split filtration of $M(H_1)$. Hence, $H_1\in
\mathcal{S}$ and $H\leqslant H_1$, which contradicts the maximality
of $H$. Therefore, $H=I$ and $M(H)=M$, whence $M$ can be filtered by
countably generated $w$-split modules.

(2) By (1), $M$ can be filtered by countably generated $w$-split
modules with a continuous chain $(M_\lambda~|~\lambda\leq\alpha)$.
For each ordinal $\lambda$, set
$M^\prime_{\lambda}=(M_{\lambda})_w$. Then by Corollary
\ref{corollary of w-split}(1), $M_\lambda$ is $w$-projective, and
hence $M^\prime_\lambda$ is a $w$-projective $w$-submodule of $M$.
By the similar proof of Lemma \ref{key lemma}(3), we have that
$M_{\lambda+1}/M_{\lambda}$ is $w$-isomorphic to
$M^\prime_{\lambda+1}/M^\prime_{\lambda}$, and so
$M^\prime_{\lambda+1}/M^\prime_{\lambda}$ is $w$-countably generated
$w$-projective module. Thus, it follows that $M$ can be filtered by
$w$-countably generated $w$-projective modules.
\end{proof}

\section*{Acknowledgments}
\noindent The authors would like to thank the referee for the
helpful comments and suggestions which substantially improved the
paper. This work was partially supported by NSFC (Nos. 11671283 and
11701398) and the Scientific Research Fund of Sichuan Provincial
Education Department (No. 17ZB0362).

\end{document}